\documentclass{amsart}

\usepackage[utf8]{inputenc}
\usepackage{verbatim}
\usepackage{graphicx}
\usepackage{graphicx,caption2,psfrag,float,color}
\usepackage{amssymb}
\usepackage{amscd}
\usepackage{amsmath}
\usepackage{mathrsfs}

\newtheorem{theorem}{Theorem}[section]

\newtheorem{lemma}[theorem]{Lemma}

\numberwithin{equation}{subsection}
\newtheorem{definition}[theorem]{Definition}

\title{The error for the second moment of cotangent sums related to the Riemann Hypothesis}
\author{Helmut Maier and Michael Th. Rassias}
\date{\today}
\address{Department of Mathematics, University of Ulm, Helmholtzstrasse 18, 89081 Ulm, Germany.}
\email{helmut.maier@uni-ulm.de}
\address{Institute of Mathematics, University of Zurich, CH-8057, Zurich, Switzerland }
\email{michail.rassias@math.uzh.ch}\thanks{}

\begin{document}

\maketitle
 
\begin{abstract} 
In various papers the authors have derived asymptotics for moments of certain cotangent sums related to the Riemann Hypothesis. S. Bettin \cite{bettin} has given an upper bound for the error term in these asymptotic results. In the present paper the authors establish a lower bound for the error term for the second moment.

\textbf{Key words:} Riemann Hypothesis, Riemann zeta function, Nyman-Beurling-B\'aez-Duarte criterion.\\
\textbf{2000 Mathematics Subject Classification:} 30C15, 11M26, 42A16, 42A20
\newline

\end{abstract}
\section{Introduction}
The authors in joint work (cf. \cite{mr, mr2, mrasympt, mr5}) and the second author in his thesis (\cite{rasthesis}) investigated the distribution of cotangent sums
$$c_0\left(\frac{r}{b}\right):=\sum_{m=1}^{b-1}\frac{m}{b}\cot\left(\frac{\pi m r}{b}\right)\:,$$
as $r$ ranges over the set
$$\{r\::\: (r, b)=1,\ A_0b\leq r\leq A_1b\}\:,\  \ \text{where }1/2<A_0<A_1<1\:.$$
These cotangent sums are related to the Estermann zeta function
$$E\left(s, \frac{r}{b},\alpha\right):=\sum_{n\geq 1}\frac{\sigma_\alpha(n)\exp(2\pi i n r/b)}{n^s}\:,$$
where $Re\:s>Re\:\alpha+1$, $b\geq 1$, $(r, b)=1$ and 
$$\sigma_{\alpha}(n):=\sum_{d\mid n} d^\alpha\:.$$
The cotangent sum $c_0(r/b)$ can be associated to the study of the Riemann Hypothesis through its relation with the Vasyunin sum $V$, which is defined by
$$V\left(\frac{r}{b}\right):=\sum_{m=1}^{b-1}\left\{\frac{mr}{b}\right\}\cot\left(\frac{\pi m r}{b}\right)\:,$$
where $\{u\}:=u-\lfloor u\rfloor$, $u\in\mathbb{R}$.\\
It can be shown that 
\[
V\left(\frac{r}{b}\right)=-c_0\left(\frac{\bar{r}}{b}\right),
\]
where $r\bar{r}\equiv 1(\bmod b)$.
We have
\begin{align*}
&\frac{1}{2\pi\sqrt{rb}}\int_{-\infty}^\infty\left|\zeta\left(\frac{1}{2}+it\right)\right|^2\left(\frac{r}{b}\right)^{it}\frac{dt}{\frac{1}{4}+t^2}\\
&\ \ \ =\frac{\log2\pi-\gamma}{2}\left(\frac{1}{r}+\frac{1}{b}\right)+\frac{b-r}{2rb}\log\frac{r}{b}-\frac{\pi}{2rb}\left(V\left(\frac{r}{b}\right)+V\left(\frac{b}{r}\right)\right)\:.
\end{align*}
The above formula is related to the Nymann-Beurling-Ba\'ez-Duarte-Vasyunin approach to the Riemann Hypothesis (see \cite{baez2, VAS}). Let
\[
d_N^2:=\inf_{D_N}\frac{1}{2\pi}\int_{-\infty}^\infty\left|1-\zeta D_N\left(\frac{1}{2}+it\right)\right|^2\frac{dt}{\frac{1}{4}+t^2}\tag{1.1}
\]
and the infimum is over all Dirichlet polynomials 
$$D_N(s):=\sum_{n=1}^{N}\frac{a_n}{n^s}\:,\ a_n\in\mathbb{C}\:,$$
of length $N$ (see \cite{bcf}).\\
The Riemann Hypothesis is true if and only if
$$\lim_{N\rightarrow+\infty} d_N=0\:.$$
The authors of the present paper in joint work (cf. \cite{mr2}), considered the moments defined by
\[
H_k:=\lim_{b\rightarrow+\infty} \phi(b)^{-1}b^{-2k}(A_1-A_0)^{-1}\sum_{\substack{A_0b\leq r\leq A_1b\\ (r,b)=1}} c_0\left(\frac{r}{b}\right)^{2k}\:,\ k\in\mathbb{N}\:,  \tag{1.1}
\]
where $\phi(\cdot)$ denotes the Euler phi-function. They could show that 
\[
H_k=\int_0^1\left(\frac{g(x)}{\pi}\right)^{2k}\: dx \:,  \tag{1.2}
\]
where
\[
g(x):=\sum_{l\geq 1}\frac{1-2\{lx\}}{l}\:   \tag{1.3}
\]
a function that has been investigated by de la Bret\`eche and Tenenbaum (\cite{bre}),
as well as Balazard and Martin (\cite{balaz1, balaz2}).\\ 
Bettin \cite{bettin} could replace the interval $(1/2, 1]$ for $A_0, A_1$ by the interval $(0,1)$. In a series of papers the authors investigated the moments $H_k$. In \cite{mrasympt} they showed:\\
Let $K\in\mathbb{N}$. There is an absolute constant $C>0$, such that
$$\int_0^1|g(x)|^K\: dx=\frac{e^\gamma}{\pi}\:\Gamma(K+1)(1+O(\exp(-CK)))\:,$$
for $K\rightarrow+\infty$.\\
In \cite{mr5} the authors could generalise this result for arbitrary positive exponents.\\
The size of the error term in (1.1) has been investigated by Bettin (\cite{bettin}). Using the Mellin transform and complex integration he could show the following result:\\
$$\frac{1}{\phi(q)}\sum_{(a,q)=1}c_0\left(\frac{a}{q}\right)^k=H_{k/2}\:q^k+O(q^{k-1+\epsilon}(Ak\log q)^{2k})\:.$$
In this paper we show that for the special case $k=2$ and $q$ a prime number Bettin's upper bound for the error term is close to best possible. Our main result is the following:
\begin{theorem}\label{main}
Let $q$ be a prime number, $H_1$ resp. $g$ be given by (1.2) resp. (1.3) and let
$$\frac{1}{q-1}\sum_{a=1}^{q-1}c_0\left(\frac{a}{q}\right)^2=H_1\:q^2+E(q)\:.$$
Then there is an absolute constant $C>0$, such that 
$$E(q)\geq Cq(\log q)^2\:,\  q\geq q_0\:.$$
\end{theorem}

\section{Continued Fractions}
We recall some fundamental definitions and results from \cite{balaz2}.
\begin{definition}\label{def21}
Let $X:=[0,1]\setminus\mathbb{Q}$ and $\alpha(x):=\{1/x\}$ for all $x\in X$, where $\{\cdot\}$ denotes the fractional part. We define the iterates of $\alpha$ by:
$$\alpha_0(x):=x,\ \alpha_k(x):=\alpha(\alpha_{k-1}(x)),\ \ \text{for all $k\in\mathbb{N}$}\:.$$
We write
$$a_0(x):=0\ \ \text{and}\ \ a_k(x):=\left\lfloor\frac{1}{\alpha_{k-1}(x)}\right\rfloor\:,\ k\geq 1\:.$$
If $x$ is irrational, then the sequence of partial fractions of $x$ is defined by the recursion
$$p_0(x):=0,\ q_0(x):=1;\ p_1(x):=1,\ q_1(x):=a_1(x)\:,$$
\begin{align*}
p_k(x)&:=a_k(x)p_{k-1}(x)+p_{k-2}(x)\:, \tag{2.1}\\
q_k(x)&:=a_k(x)q_{k-1}(x)+q_{k-2}(x)\:,\ k\geq 2\:.
\end{align*}
One writes 
\[
\frac{p_k(x)}{q_k(x)}:=[0;a_1(x),\ldots, a_k(x)]\:.\tag{2.2}
\]
The sequence 
$$\left( \frac{p_k(x)}{q_k(x)}\right)_{k=0}^{+\infty}$$ 
is called the
continued fraction expansion of $x$ and is denoted by 
$$[0;a_1(x),\ldots, a_k(x),\ldots]\:.$$
If $x$ is a rational number, then $\alpha_K(x)=0$ for some $K\in\mathbb{N}$ and we 
have:
$$x=[0;a_1(x),\ldots, a_K(x)]\:.$$
$K$ is called \textbf{the depth of} ${x}$.\\
We shall also apply the Definitions \ref{def21}, \ref{def22} for the case that the last term $a_k(x)$
is not an integer.\\
We define the functions $\beta_k$ and $\gamma_k$ by
\[
\beta_k(x):=\alpha_0(x)\alpha_1(x)\ldots \alpha_k(x)\:,\ \ (\beta_{-1}=1)\tag{2.3}
\] 
and 
\[
\gamma_k(x):=\beta_{k-1}(x)\log\frac{1}{\alpha_k(x)}\:,\tag{2.4}
\] 
with
\[
\gamma_0(x):=\log(1/x).  \tag{2.5}
\]
\begin{definition}\label{def22} \textbf{(cells)}\\
Let $k\in\mathbb{N}$, $b_0:=0$ and $b_1,\ldots, b_k\in\mathbb{N}^*=\mathbb{N}\setminus\{0\}\:.$ The \textbf{cell of depth} $k$, $\mathcal{C}(b_1,\ldots,b_k)$ is the open interval with the endpoints $[0;b_1,\ldots,b_k]$ and $[0;b_1, \ldots, b_{k-1}, b_k+1]$.\\
In the cell $\mathcal{C}(b_1,\ldots,b_k)$ the functions $a_j, p_j, q_j$ are constant for
$j\leq k$.\\ For $x\in\mathcal{C}(b_1,\ldots,b_k)$ we have:
$$a_j(x)=b_j\:,\ \frac{p_j(x)}{q_j(x)}=[0; b_1,\ldots, b_j]\:,\ j\leq k\:.$$
\end{definition}
\begin{lemma}\label{lem21}
Within the cell $\mathcal{C}(b_1,\ldots,b_k)$, $\alpha_k$ and $\gamma_k$ are differentiable 
functions of $x$. We have:
$$\alpha_k'=(-1)^k(q_k+\alpha_kq_{k-1})^2\:,$$
$$\gamma_k'=(-1)^{k-1}q_{k+1}\log\left(\frac{1}{\alpha_k}\right)+\frac{(-1)^{k-1}}{\beta_k}\:.$$
\end{lemma}
\begin{proof}
(\cite{balaz1}), Formula (34), p. 207 and (36), p. 208.
\end{proof}
\begin{lemma}\label{lem22}
$$\beta_k(x)=(-1)^{k-1}(p_k(x)-xq_k(x))=|p_k(x)-xq_k(x)|=\frac{1}{q_{k+1}(x)+\alpha_{k+1}(x)q_k(x)}\:.$$
\end{lemma}
\begin{proof}
This is formula (14) of \cite{balaz2}.
\end{proof}

\section{A representation of $g(x)$ related to Wilton's function}
We now recall the following definition from \cite{balaz2}.
The number $x$ is called a {Wilton number} if the series
\[
\sum_{k\geq 0}(-1)^k\gamma_k(x)\tag{3.1}
\]
converges.
Wilton's function $\mathcal{W}(x)$ is defined by
\[
\mathcal{W}(x):=\sum_{k\geq 0}(-1)^k\gamma_k(x)\tag{3.2}
\]
for each Wilton number $x\in (0,1)$.\\
The operator $T\::\: L^p\rightarrow L^p$ $(p>1)$ is defined by 
\[
Tf(x):=xf(\alpha(x))  \tag{3.3}
\]
For $n\in\mathbb{N}$, $x\in X$, we define
\[
\mathcal{L}(x,n):=\sum_{v=0}^n(-1)^v(T^vl)(x)\:. \tag{3.4}
\]
For $\lambda\geq 0$ we set
\[
A(\lambda):=\int_0^{+\infty} \{t\}\{\lambda t\}\:\frac{dt}{t^2} \tag{3.5}
\]
\[
F(x):= \frac{x+1}{2}\: A(1)-A(x)-\frac{x}{2}\:\log x \tag{3.6}
\]
\[
H(x):=-2\sum_{j\geq 0}(-1)^j\beta_{j-1}(x)F(\alpha_j(x))\:.  \tag{3.7}
\]
\end{definition}

\begin{lemma}\label{lem31}
We have 
\[
\mathcal{L}(x,n)=\sum_{k=0}^n (-1)^k\gamma_k(x)\:. \tag{3.8}
\]
\[
g(x)=\mathcal{L}(x,n)+H(x)+(-1)^{n+1}T^{n+1}\mathcal{W}(x).\tag{3.9}
\]
\end{lemma}
\begin{proof}
Equality (3.8) follows from (2.3)-(2.5) and (3.4).\\
Equality (3.9) follows from Lemma 2.7 of \cite{balaz1}.
\end{proof}

\section{An expression for the error-term}
We recall the following definition from \cite{balaz1}.
\begin{definition}\label{def41}
$$D_{sin}(s,x):=\sum_{n=1}^{+\infty}\frac{d(n)\sin(2\pi nx)}{n^s}\:.$$
\end{definition}
\begin{lemma}\label{lem41}
$$c_0\left(\frac{a}{q}\right)=2q\pi^{-2}D_{sin}(1, \bar{a}/q)\:,$$
where 
$$D_{sin}(1,x)=\pi g(x)\:.$$
\end{lemma}
\begin{proof}
The first fact is due to Ishibashi (\cite{ISH}), the second to  de la Bret\`eche and 
Tenenbaum \cite{bre} (see also \cite{bettin}).
\end{proof}
\begin{lemma}\label{lem42}
Let
$$\frac{1}{q-1}\sum_{a=1}^{q-1}g\left(\frac{a}{q}\right)^2=H_1+\tilde{E}(q)\:.$$
Then  Theorem \ref{main} is equivalent to 
$$\tilde{E}(q)\geq Cq^{-1}(\log q)^2\:,\ q\geq q_0,$$
for an absolute constant $C>0$.
\end{lemma}
\begin{proof}
This follows from Lemma \ref{lem41}. To estimate $\tilde{E}(q)$ we thus have to 
investigate the sums in the following.
\end{proof}
\begin{definition}\label{def41}
$$\Sigma_1:=\sum_{k\leq K}\sum_{a=1}^{q-1}\int_{\frac{a}{q}-\frac{1}{2q}}^{\frac{a}{q}+\frac{1}{2q}}\gamma_k(x)^2-\gamma_k\left(\frac{a}{q}\right)^2\: dx$$
$$\Sigma_2:=\sum_{k_1<k_2\leq K}(-1)^{k_1+k_2}\sum_{a=1}^{q-1}\int_{\frac{a}{q}-\frac{1}{2q}}^{\frac{a}{q}+\frac{1}{2q}}\left(\gamma_{k_1}(x)-\gamma_{k_1}\left(\frac{a}{q}\right)\right)\left(\gamma_{k_2}(x)-\gamma_{k_2}\left(\frac{a}{q}\right)\right)\: dx$$
$$\Sigma_3:=\sum_{k\leq K}\sum_{a=1}^{q-1}\int_{\frac{a}{q}-\frac{1}{2q}}^{\frac{a}{q}+\frac{1}{2q}}\left( H(x)-H\left(\frac{a}{q}\right)\right)\left(\gamma_{k}(x)-\gamma_{k}\left(\frac{a}{q}\right)\right)\: dx$$
$$\Sigma_4:=\sum_{k\leq K}\sum_{a=1}^{q-1}\int_{\frac{a}{q}-\frac{1}{2q}}^{\frac{a}{q}+\frac{1}{2q}}\left( H(x)-H\left(\frac{a}{q}\right)\right)^2\: dx$$
\end{definition}

\section{A lower bound for $\Sigma_1$}

We first give some facts and definitions which are also of importance in the estimate of the other terms.
\begin{lemma}\label{lem51}
Let $r$ be a rational number of depth $k$,
$$r=[0; b_1, \ldots, b_k]\:,\ k\geq 2\:.$$
Then there is exactly one pair $\mathcal{P}_k=(\mathcal{C}_1, \mathcal{C}_2)$, $\mathcal{C}_1$ a
cell of depth $k$  and $\mathcal{C}_2$ a cell of depth $k+1$, such that $r$ is a common
endpoint of both of the cells, namely
$$\mathcal{C}_1=\mathcal{C}(b_1,\ldots b_k)\ \ \text{and}\ \ \mathcal{C}_2=\mathcal{C}(b_1, \ldots, b_k-1, 2)\:.$$
\end{lemma}
\begin{proof}
By definition a cell $\tilde{\mathcal{C}}$ of depth $k+1$ that has an endpoint of depth $k$ must be of the form 
$$\tilde{\mathcal{C}}=\mathcal{C}(a_1, \ldots, a_k-1, 2),\ \ a_k\geq 2\:.$$
Thus we must have
$$\mathcal{C}_2=\mathcal{C}(b_1,\ldots b_{k-1}, 2)\:.$$
By Definition \ref{def22}, the cells of order $k$ bordering on $r$ are
$$\tilde{\mathcal{C}}=\mathcal{C}(b_1,\ldots, b_k-1)\ \ \text{and}\ \ \tilde{\mathcal{C}}=\mathcal{C}(b_1,\ldots, b_k)\:.$$
Since $\mathcal{C}(b_1,\ldots, b_k-1, 2)$ is a proper subset of $\tilde{\mathcal{C}}$,
we must have $\mathcal{C}_1=\mathcal{C}(b_1, \ldots, b_k)$.
\end{proof}
\begin{definition}\label{def51}
We call the pair $\mathcal{P}_k=(\mathcal{C}_1, \mathcal{C}_2)$ of Lemma \ref{lem51} the pair
of order $k$ of $r$.
For each $k$ we partition the set of intervals 
$$I_a:=\left[\frac{a}{q},\frac{a+1}{q}\right]$$
into two classes:
$$C_{k,1}:=\{I_a\::\: I_a\ \text{and}\ I_{a+1}\ \ \text{do not contain a rational number of depth}\ k  \}$$
$$C_{k,2}:=\{I_a\::\: I_a\ \text{or}\ I_{a+1}\ \text{contains a rational number of depth}\ k  \}\:.$$
\end{definition}
We first give a lower bound for the contribution of the intervals of class $C_{k,1}$.\\
Each $I_{a^*}\in C_{k,1}$ is entirely contained in a cell $c(I_{a^*})=\mathcal{C}(b_1, \ldots, b_k)$ of order $k$. Let
$$[b_1, \ldots, b_k]=:\frac{p_k}{q_k}\:.$$
We write $a=a_0+h$, where 
$$a_0=\min\{a\::\: I_a\subset   c(I_{a^*})\} \:. $$
We now evaluate
$$C_a=\int_{\frac{a}{q}-\frac{1}{2q}}^{\frac{a}{q}+\frac{1}{2q}}\gamma_k(x)^2-\gamma_k\left(\frac{a}{q}\right)^2\: dx$$
From Lemmas \ref{lem21} and \ref{lem22} we obtain:
$$\gamma_k'(x)=-q_{k-1}\log\left(\frac{1}{\alpha_k(x)}\right)+q_k^{-1}\left(\frac{p_k}{q_k}-x\right)^{-1}$$
 and thus 
 $$\gamma_k''(x)=2q_k^{-1}\left(\frac{p_k}{q_k}-x\right)^{-2}+O\left(\left(\frac{p_k}{q_k}-x\right)^{-1}\right)\:.$$
 We also have that 
 $$\left|\frac{p_k}{q_k}-\frac{a_0}{q}\right|\geq \frac{1}{qq_k}\:.$$
 This leads to 
 \[
 \frac{d^2}{dx^2}(\gamma_k(x)^2)=2\gamma_k'(x)^2+2\gamma_k(x)\gamma_k''(x)\geq 4q_k^{-2}\frac{q^2\log\left( \frac{q}{h}\right)}{(h+\theta(x))^2}\:,\ \ 0\leq \theta(x)\leq 1\:,   \tag{5.1}
 \]
 if $q_k\leq q^{1/3}$.\\
 By Taylor's theorem we obtain with $\theta_1(u)$, $\theta_2(u)\in(0,1),$
 \begin{align*}
 \tag{5.2} C_a&=\int_{0}^{\frac{1}{2q}}\gamma_k\left(\frac{a}{q}+u\right)^2+\gamma_k\left(\frac{a}{q}-u\right)^2-2\gamma_k\left(\frac{a}{q}\right)^2\: du\\ 
 &=\int_0^{\frac{1}{2q}} \frac{u^2}{2}\left(\frac{d^2}{dx^2}\left(\gamma_k\left(\frac{a}{q}+\theta_1(u)\right)^2\right)+\frac{d^2}{dx^2}\left(\gamma_k\left(\frac{a}{q}-\theta_2(u)\right)^2\right)\right)du\\
 &\geq c_1q_k^{-2}q^{-1}h^{-2}\log\left(\frac{q}{h} \right)
\end{align*}
for $q_k\leq q^{1/3}$ (where $c_1>0$ is an absolute constant).\\
We now investigate the contribution of the intervals $I_a\subset c_{k,2}$. We assume that
$k$ is odd. The case $k$ even is treated similarly.\\
Let $r$ be a rational number of depth $k$ in 
$$I_a=\left(\frac{a}{q}, \frac{a+1}{q}\right)\:.$$
We write
$$r=\frac{a}{q}+\frac{1}{2q}+w_0\:,\ \ w_0\in\left(-\frac{1}{2q}, \frac{1}{2q}\right)\:.$$
By Lemma \ref{lem51}, there is exactly one pair $\mathcal{P}_k$ of cells $(\mathcal{C}_1, \mathcal{C}_2)$, $\mathcal{C}_1$ of depth $k$, $\mathcal{C}_2$ of depth $k+1$,
such that $r$ is a common endpoint of both, namely
$$\mathcal{C}_1=\mathcal{C}(b_1,\ldots, b_k)\ \ \text{and}\ \ \mathcal{C}_2=\mathcal{C}(b_1, \ldots, b_k-1, 2)\:.$$
We combine the contributions of order $k$ to $I_a$ and of order $k+1$ to $I_{a+1}$,
i.e. we consider 
\begin{align*}
C(a, k)&:=\int_{\frac{a}{q}-\frac{1}{2q}}^{\frac{a}{q}+\frac{1}{2q}+w_0}\gamma_k(x)^2-\gamma_k\left(\frac{a}{q}\right)^2\: dx+\int_{\frac{a}{q}+\frac{1}{2q}+w_0}^{\frac{a+1}{q}+\frac{1}{2q}}\gamma_{k+1}(x)^2-\gamma_{k+1}\left(\frac{u+1}{q}\right)^2\: dx\\
&=: I(a, k, w_0)\:.
\end{align*}
and study $I(a, k, w_0)$ as a function of $w_0$. We first treat the case $w_0=0$.\\
For $u>0$ we write 
$$r-u=[b_1,\ldots, b_k+v]=[b_1,\ldots, b_k-1, (1+v)^{-1}]\:.$$
By Lemma \ref{lem21}, we obtain:
$$\gamma_k(r-u)-\gamma_k\left(\frac{a}{q}\right)=\left(\gamma_{k+1}(r+u)-\gamma_{k+1}\left(\frac{a+1}{q}\right)\right)(1+O(q^{-1}))\:.$$
We obtain
$$I(a,k,0)\geq c_2q_k^{-2}q^{-1}\log q\:.$$
A simple computation shows that
$$\frac{d^2I(a,k,w_0)}{dw_0^2}>0\:.$$
Thus we also have:
\[
C(a,k)\geq c_3q_k^{-2}q^{-1}\log q\:,\ \ \text{for $q_k\leq q^{1/3}$.} \tag{5.3}
\]
We still need a bound for the contribution of a cell, which is uniform in $q_k$. The width of the cell of depth $k$ with partial denominators $q_k$ is $O(1/q^2_k)$. From the bound 
$$\beta_k(x)\leq \frac{1}{q_k}\:,$$
we obtain 
\[
\int_C \gamma_k(x)^2dx=O(q_k^{-4})\:.  \tag{5.4}
\]
We now collect the estimates (5.2), (5.3), (5.4). Summing over $h$, $p_k$ and
$q_k$ we obtain 
\[
\Sigma_1\geq  c_4 q^{-1}(\log q)^2\:,  \tag{5.5}
\]
for $K$ sufficiently large.

\section{Upper bound for the other sums}
The estimate of the other sums is carried out with very similar methods. To 
estimate the sum $\Sigma_2$ - the most difficult case - we again collect pairs $\mathcal{P}_k$ of order $k$ and estimate integrals
$$\int_{-w_0}^{w_0} \gamma_{k_1}(r+v)(\gamma_{k_2}(r+v)-\gamma_{k_2+1}(r+v))\: dv\:,$$
which arise from the alternating signs in (3.8). We obtain
\[
\Sigma_i=o(q^{-1}(\log q)^2)\ \ (i=2, 3, 4)\:.  \tag{6.1}
\]
Theorem \ref{main} now follows from (5.4) and (6.1).

\end{document}